\documentclass{amsart}
\usepackage{amssymb}
\usepackage{amsfonts}

\newtheorem{theorem}{Theorem}
\theoremstyle{plain}

\newtheorem{definition}{Definition}

\newtheorem{lemma}{Lemma}

\newtheorem{remark}{Remark}

\numberwithin{equation}{section}
\input{tcilatex}

\begin{document}
\title[Simpson Type Inequalities via $\varphi -$Convexity]{Simpson Type
Inequalities via $\varphi -$Convexity}
\author{M.Emin \"{O}zdemir$^{\blacklozenge }$}
\address{$^{\blacklozenge }$Atat\"{u}rk University, K.K. Education Faculty,
Department of Mathematics, Erzurum 25240, Turkey}
\email{emos@atauni.edu.tr}
\author{Merve Avci$^{\bigstar \diamondsuit }$}
\address{$^{\bigstar }$Adiyaman University, Faculty of Science and Arts,
Department of Mathematics, Adiyaman 02040, Turkey}
\email{mavci@posta.adiyaman.edu.tr}
\thanks{$^{\diamondsuit }$Corresponding Author}
\author{A. Ocak Akdemir$^{\clubsuit }$}
\address{$^{\clubsuit }$A\u{g}ri \.{I}brahim \c{C}e\c{c}en University,
Faculty of Science and Arts, Department of Mathematics, A\u{g}ri 04100,
Turkey}

\begin{abstract}
In this paper, we obtain some Simpson type inequalities for functions whose
derivatives in absolute value are $\varphi -$convex.
\end{abstract}

\maketitle

\section{introduction and preliminaries}

Suppose $f:[a,b]\rightarrow 
\mathbb{R}
$ is a four times continuously differentiable mapping on $(a,b)$ and $%
\left\Vert f^{(4)}\right\Vert _{\infty }=\sup \left\vert
f^{(4)}(x)\right\vert <\infty .$ The following inequality%
\begin{eqnarray*}
&&\left\vert \frac{1}{3}\left[ \frac{f(a)+f(b)}{2}+2f\left( \frac{a+b}{2}%
\right) \right] -\frac{1}{b-a}\int_{a}^{b}f(x)dx\right\vert \\
&\leq &\frac{1}{2880}\left\Vert f^{(4)}\right\Vert _{\infty }\left(
b-a\right) ^{4}
\end{eqnarray*}%
is well known in the literature as Simpson's inequality.

For some results about Simpson inequality see \cite{ADS}-\cite{D}.

In \cite{ADS}, Alomari et al. proved some inequalities of Simpson type for $%
s-$convex functions by using the following Lemma.

\begin{lemma}
\label{lem 1.1} Let $f:I\subset 
\mathbb{R}
\rightarrow 
\mathbb{R}
$ be an absolutely continuous mapping on $I^{\circ }$ where $a,b\in I$ with $%
a<b.$ Then the following equality holds:%
\begin{eqnarray*}
&&\left\vert \frac{1}{6}\left[ f(a)+4f\left( \frac{a+b}{2}\right) +f(b)%
\right] -\frac{1}{b-a}\int_{a}^{b}f(x)dx\right\vert \\
&=&\left( b-a\right) \int_{0}^{1}p(t)f^{\prime }(tb+(1-t)a)dt,
\end{eqnarray*}%
where%
\begin{equation*}
p(t)=\left\{ 
\begin{array}{c}
t-\frac{1}{6},\text{ \ \ \ \ }t\in \left[ 0,\frac{1}{2}\right) \\ 
\\ 
t-\frac{5}{6},\text{ \ \ \ \ }t\in \left[ \frac{1}{2},1\right] .%
\end{array}%
\right.
\end{equation*}
\end{lemma}

Let $f,\varphi :K\rightarrow 
\mathbb{R}
,$ where $K$ is a nonempty closed set in $%
\mathbb{R}
^{n},$ be continuous functions. We recall the following results, which are
due to Noor and Noor \cite{NN}, Noor \cite{N} as follows:

\begin{definition}
\label{def 1} Let $u\in K.$ Then the set $K$ is said to be $\varphi -$convex
at $u$ with respect to $\varphi ,$ if 
\begin{equation*}
u+te^{i\varphi }(v-u)\in K,\text{ \ \ }\forall u,v\in K,\text{ \ \ }t\in
\lbrack 0,1].
\end{equation*}
\end{definition}

\begin{remark}
\label{rem 1} We would like to mention that the Definition \ref{def 1} of a $%
\varphi -$convex set has a clear geometric interpretation. This definition
essentially says that there is a path starting from a point $u$ which is
contained in $K.$ We don't require that the point $v$ should be one of the
end points of the path. This observation plays an important role in our
analysis . Note that, if we demand that $v$ should be an end point of the
path for every pair of points, $u,v\in K,$ then $e^{i\varphi }(v-u)=v-u$ if
and only if, $\varphi =0,$ and consequently $\varphi -$convexity reduces to
convexity. Thus, it is true that every convex set is also an $\varphi -$%
convex set, but the converse is not necessarily true.
\end{remark}

\begin{definition}
\label{def 2} The function $f$ on the $\varphi -$convex set $K$ is said to
be $\varphi -$convex with respect to $\varphi ,$ if 
\begin{equation*}
f(u+te^{i\varphi }(v-u))\leq (1-t)f(u)+tf(v),\text{ \ \ }\forall u,v\in K,%
\text{ \ \ }t\in \lbrack 0,1].
\end{equation*}%
The function $f$ is said to be $\varphi -$concave if and only if $-f$ is $%
\varphi -$convex. Note that every convex function is a $\varphi -$convex
function, but the converse is not true.
\end{definition}

The following inequality is known as the H\"{o}lder inequality\cite{1}:

\begin{theorem}
\label{teo 1.1} Let $p>1$ and $\frac{1}{p}+\frac{1}{q}=1.$ If $f$ and $g$
are real functions defined on $[a,b]$ and if $\left\vert f\right\vert ^{p}$
and $\left\vert g\right\vert ^{q}$ are integrable functions on $[a,b]$ then 
\begin{equation*}
\int_{a}^{b}\left\vert f(x)g(x)\right\vert dx\leq \left(
\int_{a}^{b}\left\vert f(x)\right\vert ^{p}dx\right) ^{\frac{1}{p}}\left(
\int_{a}^{b}\left\vert g(x)\right\vert ^{q}dx\right) ^{\frac{1}{q}},
\end{equation*}%
with equality holding if and only if $A\left\vert f(x)\right\vert
^{p}=B\left\vert g(x)\right\vert ^{q}$ almost everywhere, where $A$ and $B$
are constants.
\end{theorem}

\section{Simpson type inequalities for $\protect\varphi -$convex functions}

Throughout this section, let $K=[a,a+e^{i\varphi }(b-a)]$ and $0\leq \varphi
\leq \frac{\pi }{2}$

We used the following Lemma to obtain our main results.

\begin{lemma}
\label{lem 3.1} Let $K\subset 
\mathbb{R}
$ be a $\varphi -$convex subset and $f:K\rightarrow (0,\infty )$ be a
differentiable function on $K^{\circ }($the interior of $K)$, $a,b\in K$
with $a<a+e^{i\varphi }(b-a).$ If $f^{\prime }$ is integrable on $%
[a,a+e^{i\varphi }(b-a)],$ following equality holds:%
\begin{eqnarray*}
&&\left\vert \frac{1}{6}\left[ f(a)+4f\left( \frac{2a+e^{i\varphi }(b-a)}{2}%
\right) +f(a+e^{i\varphi }(b-a))\right] -\frac{1}{e^{i\varphi }(b-a)}%
\int_{a}^{a+e^{i\varphi }(b-a)}f(x)dx\right\vert \\
&=&e^{i\varphi }(b-a)\int_{0}^{1}p(t)f^{\prime }(a+te^{i\varphi }(b-a))dt,
\end{eqnarray*}%
where%
\begin{equation*}
\ \ p(t)=\left\{ 
\begin{array}{c}
t-\frac{1}{6},\text{ \ \ \ \ }t\in \left[ 0,\frac{1}{2}\right) \\ 
\\ 
t-\frac{5}{6},\text{ \ \ \ \ }t\in \left[ \frac{1}{2},1\right] .%
\end{array}%
\right.
\end{equation*}
\end{lemma}

\begin{proof}
Since $K$ is a $\varphi -$convex set, for $a,b\in K$ and $t\in \lbrack 0,1]$
we have $a+e^{i\varphi }(b-a)\in K.$ Integrating by parts implies that%
\begin{eqnarray*}
&&\int_{0}^{\frac{1}{2}}\left( t-\frac{1}{6}\right) f^{\prime
}(a+te^{i\varphi }(b-a))dt+\int_{\frac{1}{2}}^{1}\left( t-\frac{5}{6}\right)
f^{\prime }(a+te^{i\varphi }(b-a))dt \\
&=&\left. \left( t-\frac{1}{6}\right) \frac{f(a+te^{i\varphi }(b-a))}{%
e^{i\varphi }(b-a)}\right\vert _{0}^{\frac{1}{2}}-\int_{0}^{\frac{1}{2}}%
\frac{f(a+te^{i\varphi }(b-a))}{e^{i\varphi }(b-a)}dt \\
&&+\left. \left( t-\frac{5}{6}\right) \frac{f(a+te^{i\varphi }(b-a))}{%
e^{i\varphi }(b-a)}\right\vert _{\frac{1}{2}}^{1}-\int_{\frac{1}{2}}^{1}%
\frac{f(a+te^{i\varphi }(b-a))}{e^{i\varphi }(b-a)}dt \\
&=&\frac{1}{6e^{i\varphi }(b-a)}\left[ f(a)+4f\left( \frac{2a+e^{i\varphi
}(b-a)}{2}\right) +f(a+e^{i\varphi }(b-a))\right] \\
&&-\frac{1}{e^{i\varphi }(b-a)}\left[ \int_{0}^{\frac{1}{2}}f(a+te^{i\varphi
}(b-a))dt+\int_{\frac{1}{2}}^{1}f(a+te^{i\varphi }(b-a))dt\right] .
\end{eqnarray*}%
If we change the variable $x=a+te^{i\varphi }(b-a)$ and multiply the
resulting equality with $e^{i\varphi }(b-a)$ we get the desired result.
\end{proof}

\begin{theorem}
\label{teo 3.1} Let $f:K\rightarrow (0,\infty )$ be a differentiable
function on $K^{\circ }.$ If $\left\vert f^{\prime }\right\vert $ is $%
\varphi -$convex function on $K^{\circ }$ and $a,b\in K$ with $%
a<a+e^{i\varphi }(b-a).$ Then , the following inequality holds:%
\begin{eqnarray*}
&&\left\vert \frac{1}{6}\left[ f(a)+4f\left( \frac{2a+e^{i\varphi }(b-a)}{2}%
\right) +f(a+e^{i\varphi }(b-a))\right] -\frac{1}{e^{i\varphi }(b-a)}%
\int_{a}^{a+e^{i\varphi }(b-a)}f(x)dx\right\vert \\
&\leq &\frac{5}{72}e^{i\varphi }(b-a)\left[ \left\vert f^{\prime
}(a)\right\vert +\left\vert f^{\prime }(b)\right\vert \right] .
\end{eqnarray*}
\end{theorem}

\begin{proof}
From Lemma \ref{lem 3.1} and using the $\varphi -$convexity of $\left\vert
f^{\prime }\right\vert $ we have%
\begin{eqnarray*}
&&\left\vert \frac{1}{6}\left[ f(a)+4f\left( \frac{2a+e^{i\varphi }(b-a)}{2}%
\right) +f(a+e^{i\varphi }(b-a))\right] -\frac{1}{e^{i\varphi }(b-a)}%
\int_{a}^{a+e^{i\varphi }(b-a)}f(x)dx\right\vert \\
&\leq &e^{i\varphi }(b-a)\left\{ \int_{0}^{\frac{1}{2}}\left\vert t-\frac{1}{%
6}\right\vert \left\vert f^{\prime }(a+te^{i\varphi }(b-a))\right\vert
dt+\int_{\frac{1}{2}}^{1}\left\vert t-\frac{5}{6}\right\vert \left\vert
f^{\prime }(a+te^{i\varphi }(b-a))\right\vert dt\right\} \\
&\leq &e^{i\varphi }(b-a)\left\{ \int_{0}^{\frac{1}{6}}\left( \frac{1}{6}%
-t\right) \left[ \left( 1-t\right) \left\vert f^{\prime }(a)\right\vert
+t\left\vert f^{\prime }(b)\right\vert \right] dt\right. \\
&&+\left. \int_{\frac{1}{6}}^{\frac{1}{2}}\left( t-\frac{1}{6}\right) \left[
\left( 1-t\right) \left\vert f^{\prime }(a)\right\vert +t\left\vert
f^{\prime }(b)\right\vert \right] dt\right. \\
&&+\left. \int_{\frac{1}{2}}^{\frac{5}{6}}\left( \frac{5}{6}-t\right) \left[
\left( 1-t\right) \left\vert f^{\prime }(a)\right\vert +t\left\vert
f^{\prime }(b)\right\vert \right] dt\right. \\
&&+\left. \int_{\frac{5}{6}}^{1}\left( t-\frac{5}{6}\right) \left[ \left(
1-t\right) \left\vert f^{\prime }(a)\right\vert +t\left\vert f^{\prime
}(b)\right\vert \right] dt\right. \\
&=&\frac{5}{72}e^{i\varphi }(b-a)\left[ \left\vert f^{\prime }(a)\right\vert
+\left\vert f^{\prime }(b)\right\vert \right]
\end{eqnarray*}%
which completes the proof.
\end{proof}

\begin{theorem}
\label{teo 3.2} Let $f:K\rightarrow (0,\infty )$ be a differentiable
function on $K^{\circ },$ $a,b\in K$ with $a<a+e^{i\varphi }(b-a).$ If $%
\left\vert f^{\prime }\right\vert ^{q}$ is $\varphi -$convex function on $%
K^{\circ }$ for some fixed $q>1$ then the following inequality holds%
\begin{eqnarray*}
&&\left\vert \frac{1}{6}\left[ f(a)+4f\left( \frac{2a+e^{i\varphi }(b-a)}{2}%
\right) +f(a+e^{i\varphi }(b-a))\right] -\frac{1}{e^{i\varphi }(b-a)}%
\int_{a}^{a+e^{i\varphi }(b-a)}f(x)dx\right\vert \\
&\leq &e^{i\varphi }(b-a)\left( \frac{1+2^{p+1}}{6^{p+1}(p+1)}\right) ^{%
\frac{1}{p}} \\
&&\times \left\{ \left( \frac{3}{8}\left\vert f^{\prime }(a)\right\vert ^{q}+%
\frac{1}{8}\left\vert f^{\prime }(b)\right\vert ^{q}\right) ^{\frac{1}{q}%
}+\left( \frac{1}{8}\left\vert f^{\prime }(a)\right\vert ^{q}+\frac{3}{8}%
\left\vert f^{\prime }(b)\right\vert ^{q}\right) ^{\frac{1}{q}}\right\}
\end{eqnarray*}%
where $p=\frac{q}{q-1}.$
\end{theorem}

\begin{proof}
From Lemma \ref{lem 3.1} and using the H\"{o}lder inequality, we have 
\begin{eqnarray*}
&&\left\vert \frac{1}{6}\left[ f(a)+4f\left( \frac{2a+e^{i\varphi }(b-a)}{2}%
\right) +f(a+e^{i\varphi }(b-a))\right] -\frac{1}{e^{i\varphi }(b-a)}%
\int_{a}^{a+e^{i\varphi }(b-a)}f(x)dx\right\vert \\
&\leq &e^{i\varphi }(b-a)\left\{ \left( \int_{0}^{\frac{1}{2}}\left\vert t-%
\frac{1}{6}\right\vert ^{p}dt\right) ^{\frac{1}{p}}\left( \int_{0}^{\frac{1}{%
2}}\left\vert f^{\prime }(a+te^{i\varphi }(b-a))\right\vert ^{q}dt\right) ^{%
\frac{1}{q}}\right. \\
&&+\left. \left( \int_{\frac{1}{2}}^{1}\left\vert t-\frac{5}{6}\right\vert
^{p}dt\right) ^{\frac{1}{p}}\left( \int_{\frac{1}{2}}^{1}\left\vert
f^{\prime }(a+te^{i\varphi }(b-a))\right\vert ^{q}dt\right) ^{\frac{1}{q}%
}\right\} .
\end{eqnarray*}%
Since $\left\vert f^{\prime }\right\vert ^{q}$ is $\varphi -$convex, we
obtain 
\begin{eqnarray*}
&&\left\vert \frac{1}{6}\left[ f(a)+4f\left( \frac{2a+e^{i\varphi }(b-a)}{2}%
\right) +f(a+e^{i\varphi }(b-a))\right] -\frac{1}{e^{i\varphi }(b-a)}%
\int_{a}^{a+e^{i\varphi }(b-a)}f(x)dx\right\vert \\
&\leq &e^{i\varphi }(b-a)\left\{ \left( \int_{0}^{\frac{1}{6}}\left( \frac{1%
}{6}-t\right) ^{p}dt+\int_{\frac{1}{6}}^{\frac{1}{2}}\left( t-\frac{1}{6}%
\right) ^{p}dt\right) ^{\frac{1}{p}}\right. \\
&&\times \left. \left( \int_{0}^{\frac{1}{2}}\left[ \left( 1-t\right)
\left\vert f^{\prime }(a)\right\vert ^{q}+t\left\vert f^{\prime
}(b)\right\vert ^{q}\right] dt\right) ^{\frac{1}{q}}\right. \\
&&+\left. \left( \int_{\frac{1}{2}}^{\frac{5}{6}}\left( \frac{5}{6}-t\right)
^{p}dt+\int_{\frac{5}{6}}^{1}\left( t-\frac{5}{6}\right) ^{p}dt\right) ^{%
\frac{1}{p}}\right. \\
&&\times \left. \left( \int_{\frac{1}{2}}^{1}\left[ \left( 1-t\right)
\left\vert f^{\prime }(a)\right\vert ^{q}+t\left\vert f^{\prime
}(b)\right\vert ^{q}\right] dt\right) ^{\frac{1}{q}}\right\} \\
&=&e^{i\varphi }(b-a)\left( \frac{1+2^{p+1}}{6^{p+1}(p+1)}\right) ^{\frac{1}{%
p}} \\
&&\times \left\{ \left( \frac{3}{8}\left\vert f^{\prime }(a)\right\vert ^{q}+%
\frac{1}{8}\left\vert f^{\prime }(b)\right\vert ^{q}\right) ^{\frac{1}{q}%
}+\left( \frac{1}{8}\left\vert f^{\prime }(a)\right\vert ^{q}+\frac{3}{8}%
\left\vert f^{\prime }(b)\right\vert ^{q}\right) ^{\frac{1}{q}}\right\}
\end{eqnarray*}%
which is the desired.
\end{proof}

\begin{theorem}
\label{teo 3.3} Under the assumptions of Theorem \ref{teo 3.2}, we have the
following inequality%
\begin{eqnarray*}
&&\left\vert \frac{1}{6}\left[ f(a)+4f\left( \frac{2a+e^{i\varphi }(b-a)}{2}%
\right) +f(a+e^{i\varphi }(b-a))\right] -\frac{1}{e^{i\varphi }(b-a)}%
\int_{a}^{a+e^{i\varphi }(b-a)}f(x)dx\right\vert \\
&\leq &e^{i\varphi }(b-a)\left( \frac{2(1+2^{p+1})}{6^{p+1}(p+1)}\right) ^{%
\frac{1}{p}}\left[ \frac{\left\vert f^{\prime }(a)\right\vert
^{q}+\left\vert f^{\prime }(b)\right\vert ^{q}}{2}\right] ^{\frac{1}{q}}.
\end{eqnarray*}
\end{theorem}

\begin{proof}
From Lemma \ref{lem 3.1}, $\varphi -$convexity of $\left\vert f^{\prime
}\right\vert ^{q}$ and using the H\"{o}lder inequality, we have 
\begin{eqnarray*}
&&\left\vert \frac{1}{6}\left[ f(a)+4f\left( \frac{2a+e^{i\varphi }(b-a)}{2}%
\right) +f(a+e^{i\varphi }(b-a))\right] -\frac{1}{e^{i\varphi }(b-a)}%
\int_{a}^{a+e^{i\varphi }(b-a)}f(x)dx\right\vert \\
&\leq &e^{i\varphi }(b-a)\left[ \int_{0}^{1}\left\vert p(t)\right\vert
\left\vert f^{\prime }(a+te^{i\varphi }(b-a))\right\vert dt\right] \\
&\leq &e^{i\varphi }(b-a)\left( \int_{0}^{1}\left\vert p(t)\right\vert
^{p}dt\right) ^{\frac{1}{p}}\left( \int_{0}^{1}\left\vert f^{\prime
}(a+te^{i\varphi }(b-a))\right\vert ^{q}dt\right) ^{\frac{1}{q}} \\
&\leq &e^{i\varphi }(b-a)\left( \int_{0}^{\frac{1}{2}}\left\vert t-\frac{1}{6%
}\right\vert ^{p}dt+\int_{\frac{1}{2}}^{1}\left\vert t-\frac{5}{6}%
\right\vert ^{p}dt\right) ^{\frac{1}{p}}\left( \int_{0}^{1}\left[ \left(
1-t\right) \left\vert f^{\prime }(a)\right\vert ^{q}+t\left\vert f^{\prime
}(b)\right\vert ^{q}\right] dt\right) ^{\frac{1}{q}} \\
&=&e^{i\varphi }(b-a)\left( \frac{2(1+2^{p+1})}{6^{p+1}(p+1)}\right) ^{\frac{%
1}{p}}\left[ \frac{\left\vert f^{\prime }(a)\right\vert ^{q}+\left\vert
f^{\prime }(b)\right\vert ^{q}}{2}\right] ^{\frac{1}{q}}
\end{eqnarray*}%
where we used the fact that%
\begin{equation*}
\int_{0}^{\frac{1}{2}}\left\vert t-\frac{1}{6}\right\vert ^{p}dt=\int_{\frac{%
1}{2}}^{1}\left\vert t-\frac{5}{6}\right\vert ^{p}dt=\frac{(1+2^{p+1})}{%
6^{p+1}(p+1)}.
\end{equation*}%
The proof is completed.
\end{proof}

\begin{theorem}
\label{teo 3.4} Let $f:K\rightarrow (0,\infty )$ be a differentiable
function on $K^{\circ },$ $a,b\in K$ with $a<a+e^{i\varphi }(b-a).$ If $%
\left\vert f^{\prime }\right\vert ^{q}$ is $\varphi -$convex function on $%
K^{\circ }$ for some fixed $q\geq 1$ then the following inequality holds%
\begin{eqnarray*}
&&\left\vert \frac{1}{6}\left[ f(a)+4f\left( \frac{2a+e^{i\varphi }(b-a)}{2}%
\right) +f(a+e^{i\varphi }(b-a))\right] -\frac{1}{e^{i\varphi }(b-a)}%
\int_{a}^{a+e^{i\varphi }(b-a)}f(x)dx\right\vert \\
&\leq &e^{i\varphi }(b-a)\left( \frac{5}{72}\right) ^{1-\frac{1}{q}} \\
&&\times \left\{ \left( \frac{61\left\vert f^{\prime }(a)\right\vert
^{q}+29\left\vert f^{\prime }(b)\right\vert ^{q}}{1296}\right) ^{\frac{1}{q}%
}+\left( \frac{29\left\vert f^{\prime }(a)\right\vert ^{q}+61\left\vert
f^{\prime }(b)\right\vert ^{q}}{1296}\right) ^{\frac{1}{q}}\right\} .
\end{eqnarray*}
\end{theorem}

\begin{proof}
From Lemma \ref{lem 3.1} and using the power-mean inequality, we have%
\begin{eqnarray*}
&&\left\vert \frac{1}{6}\left[ f(a)+4f\left( \frac{2a+e^{i\varphi }(b-a)}{2}%
\right) +f(a+e^{i\varphi }(b-a))\right] -\frac{1}{e^{i\varphi }(b-a)}%
\int_{a}^{a+e^{i\varphi }(b-a)}f(x)dx\right\vert \\
&\leq &e^{i\varphi }(b-a) \\
&&\times \left\{ \left( \int_{0}^{\frac{1}{2}}\left\vert t-\frac{1}{6}%
\right\vert ^{1-\frac{1}{q}}dt\right) ^{\frac{1}{q}}\left( \int_{0}^{\frac{1%
}{2}}\left\vert t-\frac{1}{6}\right\vert \left\vert f^{\prime
}(a+te^{i\varphi }(b-a))\right\vert ^{q}dt\right) ^{\frac{1}{q}}\right. \\
&&+\left. \left( \int_{\frac{1}{2}}^{1}\left\vert t-\frac{5}{6}\right\vert
^{1-\frac{1}{q}}dt\right) ^{\frac{1}{q}}\left( \int_{\frac{1}{2}%
}^{1}\left\vert t-\frac{5}{6}\right\vert \left\vert f^{\prime
}(a+te^{i\varphi }(b-a))\right\vert ^{q}dt\right) ^{\frac{1}{q}}\right\} .
\end{eqnarray*}%
Since $\left\vert f^{\prime }\right\vert ^{q}$ is $\varphi -$convex function
we have%
\begin{eqnarray*}
&&\int_{0}^{\frac{1}{2}}\left\vert t-\frac{1}{6}\right\vert \left\vert
f^{\prime }(a+te^{i\varphi }(b-a))\right\vert ^{q}dt \\
&\leq &\int_{0}^{\frac{1}{6}}\left( \frac{1}{6}-t\right) \left[ \left(
1-t\right) \left\vert f^{\prime }(a)\right\vert ^{q}+t\left\vert f^{\prime
}(b)\right\vert ^{q}\right] dt \\
&&+\int_{\frac{1}{6}}^{\frac{1}{2}}\left( t-\frac{1}{6}\right) \left[ \left(
1-t\right) \left\vert f^{\prime }(a)\right\vert ^{q}+t\left\vert f^{\prime
}(b)\right\vert ^{q}\right] dt \\
&=&\frac{61\left\vert f^{\prime }(a)\right\vert ^{q}+29\left\vert f^{\prime
}(b)\right\vert ^{q}}{1296}
\end{eqnarray*}%
and 
\begin{eqnarray*}
&&\int_{\frac{1}{2}}^{1}\left\vert t-\frac{5}{6}\right\vert \left\vert
f^{\prime }(a+te^{i\varphi }(b-a))\right\vert ^{q}dt \\
&\leq &\int_{\frac{1}{2}}^{\frac{5}{6}}\left( \frac{5}{6}-t\right) \left[
\left( 1-t\right) \left\vert f^{\prime }(a)\right\vert ^{q}+t\left\vert
f^{\prime }(b)\right\vert ^{q}\right] dt \\
&&+\int_{\frac{5}{6}}^{1}\left( t-\frac{5}{6}\right) \left[ \left(
1-t\right) \left\vert f^{\prime }(a)\right\vert ^{q}+t\left\vert f^{\prime
}(b)\right\vert ^{q}\right] dt \\
&=&\frac{29\left\vert f^{\prime }(a)\right\vert ^{q}+61\left\vert f^{\prime
}(b)\right\vert ^{q}}{1296}.
\end{eqnarray*}%
Combining all the above inequalities gives us the desired result.$\ \ \ \ \
\ \ \ \ \ \ \ \ \ \ \ \ \ \ \ \ \ \ \ \ \ \ \ \ \ \ \ \ \ \ \ \ \ \ \ \ \ \
\ \ \ \ \ \ \ \ \ \ \ \ \ \ \ \ \ \ \ \ \ \ \ \ \ \ \ \ \ \ \ \ \ \ \ \ \ \
\ \ \ \ $
\end{proof}


\begin{thebibliography}{9}
\bibitem{NN} K. Inayat Noor and M. Aslam Noor, Relaxed strongly nonconvex
functions, Appl. Math. E-Notes, 6 (2006), 259-267.

\bibitem{N} M. Aslam Noor, Some new classes of nonconvex functions, Nonl.
Funct. Anal. Appl., 11 (2006), 165-171.

\bibitem{ADS} M. Alomari, M. Darus and S.S. Dragomir, New inequalities of
Simpson's type for $s-$convex functions with applications, RGMIA Res. Rep.
Coll., 12 (4) (2009).

\bibitem{AD} M. Alomari and M. Darus, \textquotedblleft On some inequalities
of Simpson-type via quasi-convex functions and
applications,\textquotedblright\ Transylvanian Journal of Mathematics and
Mechanics, vol. 2, no. 1, pp. 15--24, 2010.

\bibitem{SSO} M.Z. Sarikaya, E. Set and M.E. \"{O}zdemir, On new
inequalities of Simpson's type for $s-$convex functions, Comput. Math. Appl.
60 (2010) 2191--2199.

\bibitem{BBS} A. Barani, S. Barani and S.S. Dragomir, Simpson's Type
Inequalities for Functions Whose Third Derivatives in the Absolute Values
are $P-$Convex , RGMIA Res. Rep. Coll., 14 (2011) Preprints.

\bibitem{D} S. S. Dragomir, R. P. Agarwal, and P. Cerone, On Simpson's
inequality and applications, Journal of Inequalities and Applications, vol.
5, no. 6, pp. 533--579, 2000.

\bibitem{1} D.S. Mitrinovi\'{c}, J.E. Pe\v{c}ari\'{c} and A. M. Fink,
Classical and New Inequalities in Analysis, Kluwer Academic Publishers,
Dordrecht/Boston/London, 1993.
\end{thebibliography}
\end{document}